\documentclass[journal]{IEEEtran}
\usepackage{amssymb}
\usepackage{color}
\usepackage[dvipdfmx]{graphicx,xcolor}
\usepackage{framed}
\usepackage{bm}
\usepackage{comment}
\usepackage[verbose]{cite}
\usepackage{algorithm}
\usepackage{algorithmic}
\newtheorem{theorem}{Theorem}

\newtheorem{remark}{Remark}
\newtheorem{lemma}{Lemma}

\newtheorem{proposition}{Proposition}

\newtheorem{proof}{Proof}

\newcommand{\bb}{\mathbb}
\usepackage{latexsym}
\def\qed{\hfill $\Box$} 

\usepackage{ascmac}

\ifCLASSINFOpdf
\else
\fi
\usepackage{amsmath}
\begin{document}
%
\title{Reduced model reconstruction method for\\ stable positive network systems}
%
%
%

\author{Kazuhiro~Sato
\thanks{K. Sato is with the Department of Mathematical Informatics, 
Graduate School of Information Science and Technology, The University of Tokyo,
 Tokyo 113-8656, Japan,
email: kazuhiro@mist.i.u-tokyo.ac.jp}
}

\maketitle
\thispagestyle{empty}
\pagestyle{empty}

\begin{abstract}
We consider a reconstruction problem of a reduced stable positive network system with the preservation of the original interconnection structure based on an $H^2$ optimal model reduction problem with constraints.
To this end, we define an important set using the Perron--Frobenius theory of nonnegative matrices such that
all elements of the set are stable and Metzler.
Using the projection onto the set,
we propose a cyclic projected gradient method to produce a better reduced model than an initial reduced model in the sense of the $H^2$ norm.
In the method,
we use Lipschitz constants of the gradients of our objective function to define
the step sizes without a line search method whose computational complexity is large.
Moreover, the existence of the Lipschitz constants guarantees the global convergence  of our proposed algorithm to a stationary point.
 The numerical experiments demonstrate that the proposed algorithm improves a given reduced model, and can be used for large-scale systems.
\end{abstract}


%
\IEEEpeerreviewmaketitle

\section{Introduction} \label{Sec1}
%
%
%
%
\IEEEPARstart{M}{odel} reduction of stable positive network systems is one of the most important {topics}, because
the systems can model biological systems \cite{bernstein1993compartmental, gu2015controllability, hernandez2011discrete, yan2017network}, interconnected systems \cite{ebihara2017analysis, haddad2010nonnegative}, and multi-agent systems \cite{mesbahi2010graph}.
In particular, the preservation of the original network interconnection structure and positivity is useful to provide a physical interpretation for reduced systems.
For {this} reason, the reduction methods of preserving them have been proposed based on clustering methods \cite{besselink2016clustering, cheng2020clustering, cheng2021model, ishizaki2014model, monshizadeh2014projection} and Kron reduction methods \cite{dorfler2013kron, sugiyama2022kron}.
Although these methods preserve them, they do not guarantee the $H^2$-optimality.
That is, it may be possible to improve reduced systems obtained by the methods in the sense of the $H^2$ norm.

For improving the $H^2$ norm performance {of a reduced stable positive network system, to the best of our knowledge, only}
\cite{misawa2022h}
proposed a reconstruction method using {the Riemannian augmented Lagrangian method (RALM)} with equality and inequality constraints.
More concretely, in \cite{misawa2022h},
a Riemannian manifold constraint introduced in \cite{sato2019riemannian} is used to always guarantee
 the stability of reduced systems.
 Furthermore, equality and inequality constraints in \cite{misawa2022h} aim to preserve the positive property and interconnection structure of original systems.
However, the algorithm proposed in \cite{misawa2022h}
may not always produce a positive reduced system with the original interconnection structure, because 
it is based on the augmented Lagrangian method \cite{liu2019simple}, which may generate an infeasible solution.

To always obtain a stable positive reduced network system with the original interconnection structure,
in this paper, we formulate a novel optimization problem for developing an effective reconstruction method.
The set of feasible solutions to the novel problem is narrower than that of the problem in \cite{misawa2022h}.
However,
the feasible set in this paper is preferable compared with that of \cite{misawa2022h} for reducing large-scale systems,
because 
the projection onto the set
can be easily calculated.
In fact, using the projection, we propose a cyclic projected gradient method that always generates a feasible solution unlike the method developed in \cite{misawa2022h} and provides a better reduced model than an initial reduced model in the sense of the $H^2$ norm.

The contributions of this paper are summarized as follows.
\begin{itemize}
\item[(i)]
 Using an initial reduced state transition matrix with the stability and Metzler properties,
we define an important set based on the Perron--Frobenius theory of nonnegative matrices.
We prove that all elements of the set are stable and Metzler.

\item[(ii)]
Using the projection onto the set as mentioned in (i),
we propose a cyclic projected gradient method to produce a better stable positive reduced network system of preserving the original interconnection structure than an initial reduced model in the sense of the $H^2$ norm. 
To this end, we derive three Lipschitz constants of the gradients of our objective function in terms of three variables, and use
those to define
step sizes without a line search method.
This is practically important for reducing large-scale systems, because the computational complexity of a line search method for determining step sizes is large, as explained in Remark \ref{remark_line_search} in Section \ref{Sec5}.
Moreover, we show the global convergence property of our proposed algorithm to a stationary point of our optimization problem {and the effectiveness of the algorithm with comparisons to the RALM-based reduction method proposed in \cite{misawa2022h}.}
\end{itemize}

The remainder of this paper is organized as follows.
In Section \ref{Sec2}, we describe assumptions in this paper, define an initial reduced stable positive network model, explain the optimization problem considered in 
\cite{misawa2022h},
and point out the difficulty of the problem.
In Section \ref{Sec3}, we 
define the important set based on the Perron--Frobenius theory of nonnegative matrices and formulate a novel optimization problem using the set.
In Section \ref{Sec4}, we
derive three Lipschitz constants of the gradients of our objective function with respect to three variables.
In Section \ref{Sec5}, we
propose a cyclic projected gradient method for solving our optimization problem and prove the global convergence to a stationary point of the problem.
In Section \ref{Sec6}, we demonstrate the effectiveness of the proposed method.
Finally, our conclusions and extendability to semi-stable systems are presented in Section \ref{Sec7}.

{\it Notation:} 
 The sets of real, nonnegative, and complex numbers are denoted by ${\bb R}$, ${\bb R}_{\geq 0}$, and ${\bb C}$, respectively.
 For matrices $A=(a_{ij}), B=(b_{ij})\in {\bb R}^{m\times n}$, $A\leq B$ means that for all $(i,j)$, $a_{ij}\leq b_{ij}$ holds.
For a matrix $A\in {\bb R}^{m\times n}$, $\|A\|_{\rm F}$ denotes the Frobenius norm of $A$; i.e.,
$\|A\|_{\rm F}:=\sqrt{{\rm tr}(A^{\top}A)}$,
where
the superscript $A^\top$ denotes the transpose of $A$, and ${\rm tr}(M)$ denotes the sum of the diagonal elements of {a square matrix} $M$.
Given a vector $v\in {\bb R}^n$, $\|v\|_2$ denotes the usual Euclidean norm.
The $L^2$ space on $\mathbb{R}^m$ is denoted by $L^2({\mathbb R}^m)$ with the norm $\|f\|_{L^2}:= \sqrt{\int_0^{\infty} \|f(t)\|_2^2 dt}$, where $f:{\bb R}_{\geq 0}\rightarrow \mathbb {R}^m$ is a measurable function.
For a matrix $G(s)\in {\bb C}^{p\times m}$ without poles in the closed right half-plane in ${\bb C}$, the $H^2$ norm of $G$ is defined as
$\|G\|_{H^2} := \sqrt{ \frac{1}{2\pi} \int_{-\infty}^{\infty} \|G({\rm i}\omega)\|_{\rm F}^2 d\omega }$,
where ${\rm i}$ is the imaginary unit.
The symbol $I_n\in {\bb R}^{n\times n}$ denotes the identity matrix. 



\section{Preliminaries} \label{Sec2}

\subsection{Assumptions}

The original large-scale network system in this paper is modeled as 
\begin{align}
\begin{cases}
\dot{x}(t) = Ax(t) + Bu(t), \\
y(t) = Cx(t),
\end{cases} \label{original}
\end{align}
 with the state $x(t)\in {\bb R}^n$, input $u(t)\in {\bb R}^m$, output $y(t)\in {\bb R}^p$, and appropriate size constant real matrices $A, B, C$.
For system \eqref{original},
we impose the following assumptions:
\begin{enumerate}
    \item The matrix $A$ is  stable.
That is, the real parts of all the eigenvalues of matrix $A$ are negative. 
In this case, system (\ref{original}) is called asymptotically stable.

\item The matrix $A$ is a Metzler matrix, which means that
  every off-diagonal entry of $A$ is nonnegative. Moreover, the matrices $B$ and $C$ are nonnegative.
That is, not only the output $y(t)$ but also the state $x(t)$ is nonnegative with the nonnegative input $u(t)$ and initial state $x(0)$.
\end{enumerate}

We call system \eqref{original} with the assumptions 1) and 2) an asymptotically stable positive network, which we abbreviate ASPN throughout this paper.
Moreover, we denote the original network graph by $\mathcal{G}=(\mathcal{V}, \mathcal{E})$,
where $\mathcal{V}:=\{1,\ldots, n\}$ is the node set and $\mathcal{E}\subset \mathcal{V}\times \mathcal{V}$ is the edge set defined by
the nonzero entries of the matrix $A$.
{Note that $\mathcal{G}$ may be a directed graph with self-loops.
Furthermore, note that
ASPN \eqref{original} can be regarded as a linearized system of nonlinear positive network systems such as biological systems \cite{bernstein1993compartmental, gu2015controllability, hernandez2011discrete, yan2017network}, interconnected systems \cite{ebihara2017analysis, haddad2010nonnegative}, and multi-agent systems \cite{mesbahi2010graph}, as shown in Proposition 2.10 in \cite{haddad2010nonnegative}.}

\subsection{Initial Reduced Network Model} \label{Sec2-B}

We divide the node set $\mathcal{V}$ of the original graph $\mathcal{G}$ into $r$, which is greatly smaller than $n$, nonempty and disjoint subsets denoted by $\mathcal{C}_1,\ldots, \mathcal{C}_r$ called clusters of $\mathcal{G}$.
Then, we define
the characteristic matrix $\Pi \in \mathbb{R}^{n\times r}$ of $\mathcal{C}_1,\ldots, \mathcal{C}_r$ as the binary matrix
\begin{align*}
    \Pi_{ij} :=\left\{
    \begin{array}{ll}
        1, & \text{if } i \in \mathcal{C}_j,\\
        0, & \text{otherwise}.
    \end{array}
    \right. 
\end{align*}

Using the characteristic matrix $\Pi$, we define the initial reduced system of the original system (\ref{original}) as
\begin{align}
  \label{reduced_initial}
  \begin{cases}
    \dot{x}_r (t) = A_r^{(0)} x_r (t) + B_r^{(0)} u(t)\\
    y_r (t) = C_r^{(0)} x_r (t), 
  \end{cases}
\end{align}
where $x_r(t)\in \mathbb{R}^r$, $y_r(t)\in \mathbb{R}^p$, 
and 
\begin{align}
\label{shokiten}
\begin{cases}
   A_r^{(0)} := (\Pi^\top\Pi)^{-1}\Pi^\top A \Pi - \alpha I_r, \\
    B_r^{(0)} := (\Pi^\top\Pi)^{-1}\Pi^\top B,\\
    C_r^{(0)} := C\Pi
\end{cases}
 \end{align}
with $\alpha\geq 0$.
Note that $\Pi^\top \Pi$ is invertible, because $\Pi^\top \Pi$ is a diagonal matrix whose diagonal elements are the number of nodes which compose the each cluster.
The matrix $A_r^{(0)}$ is Metzler for any $\alpha\geq 0$.
However, because the matrix $A_r^{(0)}$ is not {always} stable as shown in \cite{misawa2022h},
we need to choose $\alpha$ such that the resulting matrix $A_r^{(0)}$ is stable.
{It should be noted that if $A$ is a Laplacian matrix,
$A_r=(\Pi^\top\Pi)^{-1}\Pi^\top A \Pi$ is also a Laplacian matrix \cite{cheng2020clustering}.
That is, $A$ and $A_r$ are both semi-stable,
where system \eqref{original} is called semi-stable if
the zero eigenvalues of $A$ are semisimple and the real parts of all the other eigenvalues are negative.
In Section \ref{Sec7},
we describe an application of our proposed method to semi-stable positive network systems including a Laplacian dynamical system.}

We define
$\mathcal{G}^{(0)}_r=(\mathcal{V}^{(0)}_r, \mathcal{E}^{(0)}_r)$
as the reduced network graph associated to system \eqref{reduced_initial},
where $\mathcal{V}^{(0)}_r:=\{1,\ldots, r\}$ is the node set and $\mathcal{E}^{(0)}_r\subset \mathcal{V}^{(0)}_r\times \mathcal{V}^{(0)}_r$ is the edge set defined by
the nonzero entries of the matrix $A^{(0)}_r$.
The reduced graph $\mathcal{G}^{(0)}_r$ has the same interconnection structure with the original graph $\mathcal{G}$ {except for self-loops}.
That is,
if there is a directed path from $i \in \mathcal{V}^{(0)}_r$ to $j \in \mathcal{V}^{(0)}_r$ {$(i\neq j)$}, there is a directed path from a node of $\pi^{-1}(i)$ to a node of $\pi^{-1}(j)$ in $\mathcal{G}$.
Here, $\pi:\mathcal{V}\rightarrow \mathcal{V}^{(0)}_r$ is the associated map to $\Pi$.

The matrices $A_r^{(0)}$, $B_r^{(0)}$, and $C_r^{(0)}$ in \eqref{shokiten} define
\begin{align*}
 {\rm z} (A_r^{(0)}) &:= \left\{(i,j) \left| (A_r^{(0)})_{ij}=0\quad (i\neq j)  \right.\right\}, \\
    {\rm st}(A_r^{(0)}) &:= \left\{ A_r \in {\bb R}^{r\times r} \left|
        (A_r)_{ij}=0\,\,{\rm if}\,\, (i,j)\in {\rm z}(A_r^{(0)})
	\right. \right\}, \\
    {\rm z} (B_r^{(0)}) &:= \left\{(i,j) \left| (B_r^{(0)})_{ij}=0 \right. \right\}, \\
  {\rm st}(B_r^{(0)}) &:= \left\{ B_r \in {\bb R}^{r\times m}_{\geq 0} \left|
        (B_r)_{ij}=0\,\,{\rm if}\,\, (i,j)\in {\rm z}(B_r^{(0)})
    \right. \right\}, \\
     {\rm z} (C_r^{(0)}) &:= \left\{(i,j) \left| (C_r^{(0)})_{ij}=0 \right. \right\}, \\
{\rm st}(C_r^{(0)}) &:= \left\{ C_r \in {\bb R}^{p\times r}_{\geq 0} \left|
        (C_r)_{ij}=0\,\,{\rm if}\,\,(i,j)\in {\rm z}(C_r^{(0)})
    \right. \right\}.
\end{align*}
These sets are used to formulate our problem.

\subsection{Reconstruction Problem in \cite{misawa2022h}} \label{Sec2-C}

To reconstruct a novel ASPN
\begin{align}
\begin{cases}
\dot{x}_r(t) = A_rx_r(t) + B_ru(t), \\
y_r(t) = C_rx_r(t),
\end{cases} \label{reduced}
\end{align}
 of preserving the interconnection structure of the original graph $\mathcal{G}$ better than initial reduced model \eqref{reduced_initial} in the sense of the $H^2$ norm, we introduce an $H^2$ optimal model reduction problem using
the transfer functions of original system \eqref{original} and reduced system \eqref{reduced} defined as
\begin{align*}
G(s) :=C(sI_n-A)^{-1}B,\,\, 
G_r(s) :=C_r(sI_r-A_r)^{-1}B_r 
\end{align*}
for $s\in {\bb C}$, respectively.
This is because
\begin{align}
\sup_{t\geq 0}\|y(t)-y_r(t)\|_2 \leq \|G-G_r\|_{H^2} \label{relation_daiji}
\end{align}
holds under $\|u\|_{L^2}\leq 1$,
as explained in \cite{gugercin2008h_2, sato2019riemannian}.
Inequality \eqref{relation_daiji} indicates that
the maximum output error norm can be expected to become almost zero when
 $\|G-G_r\|_{H^2}$ is sufficiently small.

The reconstruction problem can be formulated as
\begin{framed}
\noindent
{\bf Problem 0}: Given $r<n$ and $(A_r^{(0)}, B_r^{(0)}, C_r^{(0)})$ in \eqref{shokiten},
\begin{align*}
&\min_{(A_r, B_r,C_r)}\quad \|G-G_r\|_{H^2}^2 \\
&{\rm subject\,\,to}\quad A_r\in {\rm st}(A_r^{(0)})\,\,{\rm is\,\, stable\,\, and\,\, Metzler},\\
&\quad\quad\quad\quad\quad B_r\in {\rm st}(B_r^{(0)}),\,\,C_r\in {\rm st}(C_r^{(0)}).
\end{align*}
\end{framed}

The objective function $\|G-G_r\|_{H^2}^2$ is a non-convex function of $(A_r,B_r,C_r)$,
because
 \begin{align}
 \|G-G_r\|^2_{H^2} = 2f(A_r, B_r, C_r) + \|G\|_{H^2}^2, \label{relation_H2_F}
 \end{align}
and $f(A_r,B_r, C_r)$ is a non-convex function of $(A_r,B_r,C_r)$.
Here,
 \begin{align}
f(A_r,B_r,C_r) &:=\frac{1}{2}{\rm tr}(C_r PC_r^{\top}-2C_rX^{\top}C^{\top}) \label{h}\\
&=\frac{1}{2}{\rm tr}(B_r^{\top} Q B_r+2B^{\top} Y B_r), \nonumber
\end{align}
where $X$, $Y$, $P$, and $Q$ are the solutions to the Sylvester equations
\begin{align}
AX +X A_r^{\top} +BB^{\top}_r  = 0, \label{X}\\
A^{\top}Y +YA_r -C^{\top}C_r  = 0, \label{Y}\\
A_rP +P A_r^{\top} +B_rB^{\top}_r  = 0, \label{P}\\
A_r^{\top}Q +Q A_r +C_r^{\top}C_r  = 0, \label{Q}
\end{align}
respectively.
Because \eqref{relation_H2_F} holds and
$\|G\|_{H^2}$ is independent of $(A_r, B_r, C_r)$ of reduced system \eqref{reduced},
the minimization of $\|G-G_r\|_{H^2}^2$ is equivalent to that of $f(A_r,B_r,C_r)$.
Thus, $f(A_r,B_r,C_r)$
 has been frequently used as the objective function in $H^2$ optimal model reduction problems \cite{misawa2022h, sato2018structure, sato2019riemannian}.

\begin{remark}
We defined initial model \eqref{reduced_initial} using the clustering method in Section \ref{Sec2-B}.
To define initial model \eqref{reduced_initial}, we can use other methods such as Kron reduction methods \cite{dorfler2013kron, sugiyama2022kron}.
\end{remark}

\begin{remark} \label{remark_motivation}
Although the most difficult point of Problem 0 is to ensure the stability of $A_r$,
this can be resolved using a Riemannian manifold constraint proposed in \cite{sato2019riemannian}.
In fact, using the Riemannian manifold formulation with equality and inequality constraints, \cite{misawa2022h} proposed a Riemannian augmented Lagrangian method \cite{liu2019simple} for solving Problem 0.
However, we need to carefully choose hyper parameters in the method
 to preserve the positivity and original interconnection structure.
 That is, for some applications, it may be difficult to obtain feasible solutions to Problem 0 using the method in \cite{misawa2022h}.
\end{remark}

\section{Problem Setting} \label{Sec3}

\subsection{Compact subset of stable and Metzler matrices}
As mentioned in Remark \ref{remark_motivation}, the most difficult point to develop an algorithm for solving Problem 0 is to guarantee that $A_r\in {\bb R}^{r\times r}$ in \eqref{reduced} is stable and Metzler.
To easily guarantee this, we construct a compact subset of stable and Metzler matrices of ${\bb R}^{r\times r}$ using the information of the initial matrix $A_r^{(0)}$ in \eqref{shokiten}
under the assumption that $A_r^{(0)}$ is irreducible in addition to the stable and Metzler properties.

The following lemma is based on a famous result of the Perron--Frobenius theory of nonnegative matrices, as shown in Section 8.3 in \cite{meyer2000matrix}.

\begin{lemma} \label{Lem_Perron_Frobenius}
There exists an eigenvalue $\mu_1$ of $A_r^{(0)}$ such that $\mu_1$ is a real number and 
\begin{align}
{\rm Re}(\mu_r) \leq \cdots \leq {\rm Re}(\mu_2)< \mu_1 <0   \label{eigen_ineq}
\end{align}
where $\mu_2,\ldots, \mu_r\in {\bb C}$ are also the eigenvalues of $A_r^{(0)}$. 
Moreover, the right eigenvector $v_1$ and the left eigenvector $w_1$ corresponding to $\mu_1$ can be chosen to be positive vectors satisfying 
\begin{align}
w_1^{\top}v_1 =1. \label{w_normalization}
\end{align}
\end{lemma}
\begin{proof}
See Appendix \ref{Ape_proof_Lemma}.
\qed
\end{proof}

Using Lemma \ref{Lem_Perron_Frobenius}, we obtain the following key theorem to formulate our problem.

\begin{theorem} \label{Thm_key_set}
Let $\mu_1$, $\mu_r$, $v_1$, and $w_1$ be the same symbols with the statements in Lemma \ref{Lem_Perron_Frobenius}.
That is, $\mu_1$ and $\mu_r$ are the eigenvalues of $A_r^{(0)}$ satisfying \eqref{eigen_ineq},
$v_1$ and $w_1$ are the right and left positive eigenvectors, that satisfies \eqref{w_normalization},  corresponding to $\mu_1$, respectively.
For any positive real number $\epsilon$ satisfying $\mu_1+\epsilon \leq 0$, define
\begin{align}
\bar{A}_r:= A_r^{(0)}-(\mu_1+\epsilon) v_1 w_1^{\top}. \label{bar_Ar}
\end{align}
Then, for any positive real number $\gamma$ satisfying $-\gamma I_r \leq A_r^{(0)}$,
each matrix in
\begin{align}
S_{A_r}(A_r^{(0)},\epsilon,\gamma) 
:= \{ A_r\in {\bb R}^{r\times r}\,|\, -\gamma I_r \leq A_r \leq \bar{A}_r \}
\end{align}
 is stable and Metzler, and $A_r^{(0)}\in S_{A_r}(A_r^{(0)},\epsilon,\gamma)$.
Moreover, the real parts of all the eigenvalues in any matrices in $S_{A_r}$ are less than or equal to $-\epsilon$.
\end{theorem}
\begin{proof}
See Appendix \ref{Ape_proof_Theorem}.\qed
\end{proof}

\subsection{Novel reconstruction problem} \label{Sec2-D}

From the above discussions,
we consider the following novel ASPN reconstruction problem,
which is a non-convex optimization problem with convex constraints.

\begin{framed}
\noindent
{\bf Problem 1}: Given $r<n$, $\epsilon>0$, $\gamma>0$, and $(A_r^{(0)}, B_r^{(0)}, C_r^{(0)})$ in \eqref{shokiten},
\begin{align*}
&\min_{(A_r, B_r,C_r)}\quad f(A_r,B_r,C_r) \\
&{\rm subject\,\,to}\quad A_r\in {\rm st}(A_r^{(0)})\cap S_{A_r}(A_r^{(0)},\epsilon,\gamma),\\
&\quad\quad\quad\quad\quad B_r\in {\rm st}(B_r^{(0)}),\,\,C_r\in {\rm st}(C_r^{(0)}).
\end{align*}
\end{framed}

\noindent
Problem 1 is a more tractable problem than Problem 0,
because the set ${\rm st}(A_r^{(0)})\cap S_{A_r}(A_r^{(0)},\epsilon,\gamma)$ is a simple convex set unlike the manifold constraint with inequality constraints{, which was formulated in \cite{misawa2022h},} on $A_r$ in Problem 0.

Problem 1 is equivalent to the following unconstrained non-convex optimization problem.

\begin{framed}
\noindent
{\bf Problem 1'}: Given $r<n$, $\epsilon>0$, $\gamma>0$, and $(A_r^{(0)}, B_r^{(0)}, C_r^{(0)})$ in \eqref{shokiten},
\begin{align*}
&\min_{(A_r, B_r,C_r)}\quad h(A_r,B_r, {C_r}).
\end{align*}
\end{framed}

\noindent
Here,
\begin{align*}
h(A_r,B_r, {C_r}) := & f(A_r,B_r,C_r)+g(A_r,B_r,C_r),\\
    g(A_r,B_r,C_r):=&\mathcal{I}_{{\rm st}(A_r^{(0)})\cap S_{A_r}(A_r^{(0)},\epsilon,\gamma)}(A_r)\\
    &
+\mathcal{I}_{{\rm st}(B_r^{(0)})}(B_r)+ \mathcal{I}_{{\rm st}(C_r^{(0)})}(C_r),
\end{align*}
where
 $\mathcal{I}_{S}$ denotes the indicator function of a set $S$.
That is, for an arbitrary set $S$,
\begin{align}
\mathcal{I}_{S}(X) := \begin{cases}
0\quad &{\rm if}\quad X\in S, \\
+\infty &{\rm if}\quad X\not\in S.
\end{cases} \label{indicator_func}
\end{align}
Note that
$g(A_r,B_r, C_r)$ is convex in terms of $(A_r,B_r, C_r)$
unlike $f(A_r,B_r, C_r)$.

Because Problem 1' is a non-convex optimization problem,
we develop an algorithm for finding a stationary point to Problem 1' instead of a global minimizer.
Here, a stationary point to Problem 1' is $(A^*_r, B^*_r, C^*_r) \in {\bb R}^{r\times r}\times {\bb R}^{r\times m}\times {\bb R}^{p\times r}$ satisfying
\begin{align}
0 \in \partial h(A^*_r,B^*_r, C^*_r), \label{stationary_Pro1}
\end{align}
where $\partial h(A_r,B_r,C_r)$ denotes the limiting subdifferential of $h$ at $(A_r, B_r, C_r)$.
For the {detailed} explanation of the limiting subdifferential, see \cite{li2020understanding}.

\begin{remark}
The feasible set of Problem 1 is included in that of Problem 0.
However, a stationary point to Problem 1 may be better than Problem 0 in the sense of the $H^2$ norm.
This is because the objective function is non-convex.
That is, stationary points that we can obtain for Problems 0 and 1 are highly dependent on an initial point
$(A_r^{(0)}, B_r^{(0)}, C_r^{(0)})$.
In Section \ref{Sec6}, we demonstrate this fact.
\end{remark}


\section{Theoretical analysis}  \label{Sec4}

To develop an efficient algorithm for Problem 1',
we show that there exist positive $L_{A_r}(B_r,C_r)$, $L_{B_r}(A_r, C_r)$, and $L_{C_r}(A_r, B_r)$, which {are} called the block Lipschitz constants \cite{beck2017first}, such that
\begin{align}
&\|\nabla_{A_r} f((A_r)_1,B_r,C_r) - \nabla_{A_r} f((A_r)_2,B_r,C_r)\|_{\rm F} \nonumber\\
\leq& L_{A_r}(B_r,C_r) \|(A_r)_1-(A_r)_2\|_{\rm F}, \label{Lip_alpha}\\
&\|\nabla_{B_r} f(A_r,(B_r)_1,C_r) - \nabla_{B_r} f(A_r,(B_r)_2,C_r)\|_{\rm F} \nonumber\\ 
\leq& L_{B_r}(A_r,C_r) \|(B_r)_1-(B_r)_2\|_{\rm F}, \label{Lip_B}\\
&\|\nabla_{C_r} f(A_r,B_r,(C_r)_1) - \nabla_{C_r} f(A_r,B_r,(C_r)_2)\|_{\rm F} \nonumber \\
\leq& L_{C_r}(A_r,B_r) \|(C_r)_1-(C_r)_2\|_{\rm F}, \label{Lip_C}
\end{align}
where for $i=1, 2$, $((A_r)_i, B_r, C_r)$, $(A_r,(B_r)_i,C_r)$, and $(A_r, B_r, (C_r)_i)$
are contained in $S_{A_r}(A_r^{(0)},\epsilon,\gamma)\times {\bb R}^{r\times m}\times {\bb R}^{p\times r}$.
We use the block Lipschitz constants to define step sizes in our proposed algorithms in Section \ref{Sec5} without performing a line-search, as explained in \cite{nocedal2006numerical}.
Moreover, the block Lipschitz constants are used to prove the global convergence of a sequence generated by our proposed algorithm described in Section \ref{Sec4} to a stationary point of Problem 1.

To this end, we note that
the gradients of $f$ defined by \eqref{h} in terms of $A_r$, $B_r$, and $C_r$ are given by
\begin{align}
\nabla_{A_r} f(A_r,B_r,C_r) &= QP + Y^{\top} X, \label{grad_h_A}\\
\nabla_{B_r} f(A_r,B_r,C_r) &= QB_r+Y^{\top} B, \label{grad_h_B}\\
\nabla_{C_r} f(A_r,B_r,C_r) &= C_rP-CX, \label{grad_h_C}
\end{align}
respectively, as shown in Theorem 3.3 in \cite{van2008h2} and Section 3.2 in \cite{wilson1970optimum},
where $X$, $Y$, $P$, and $Q$ are the solutions to \eqref{X}, \eqref{Y}, \eqref{P}, and \eqref{Q}, respectively.

\subsection{Proof of \eqref{Lip_B} and \eqref{Lip_C}}

Using \eqref{grad_h_B} and \eqref{grad_h_C}, the expressions of $L_{B_r}(A_r,C_r)$ and $L_{C_r}(A_r,B_r)$ can be easily derived as follows.

\begin{theorem} \label{Thm1}
If
\begin{align}
L_{B_r}(A_r, C_r) := \|Q\|_{\rm F}, \label{Lip_B2}
\end{align}
\eqref{Lip_B} holds for any $(A_r,(B_r)_1,C_r),\, (A_r,(B_r)_2,C_r)\in S_{A_r}(A_r^{(0)},\epsilon,\gamma)\times {\bb R}^{r\times m}\times {\bb R}^{p\times r}$.
Moreover,
if 
\begin{align}
L_{C_r}(A_r, B_r) := \|P\|_{\rm F}, \label{Lip_C2}
\end{align}
\eqref{Lip_C} holds for any $(A_r,B_r,(C_r)_1),\, (A_r,B_r,(C_r)_2)\in S_{A_r}(A_r^{(0)},\epsilon,\gamma)\times {\bb R}^{r\times m}\times {\bb R}^{p\times r}$.
\end{theorem}
\begin{proof}
It follows from \eqref{grad_h_B} that
 for any $(A_r,(B_r)_1,C_r),\, (A_r,(B_r)_2,C_r)\in S_{A_r}(A_r^{(0)},\epsilon,\gamma)\times {\bb R}^{r\times m}\times {\bb R}^{p\times r}$,
\begin{align*}
&\|\nabla_{B_r} f(A_r,(B_r)_1,C_r) - \nabla_{B_r} f(A_r,(B_r)_2,C_r)\|_{\rm F} \\
\leq & \|Q ( (B_r)_1- (B_r)_2)\|_{\rm F} 
\leq  \|Q\|_{\rm F} \|( (B_r)_1- (B_r)_2)\|_{\rm F}.
\end{align*}
Hence, if \eqref{Lip_B2} holds, \eqref{Lip_B} is satisfied for any $(A_r,B_r,(C_r)_1),\, (A_r,B_r,(C_r)_2)\in S_{A_r}(A_r^{(0)},\epsilon,\gamma)\times {\bb R}^{r\times m}\times {\bb R}^{p\times r}$.

Similarly, we can show that if \eqref{Lip_C2} holds,
\eqref{Lip_C} is satisfied for any $(A_r,B_r,(C_r)_1),\, (A_r,B_r,(C_r)_2)\in S_{A_r}(A_r^{(0)},\epsilon,\gamma)\times {\bb R}^{r\times m}\times {\bb R}^{p\times r}$. \qed
\end{proof}

The matrices $P$ and $Q$ in \eqref{Lip_C2} and \eqref{Lip_B2} are the solutions to
Lyapunov equations \eqref{P} and \eqref{Q}, respectively.
That is, $P$ is the controllability Gramian, which is a function of $A_r$ and $B_r$, of reduced system \eqref{reduced} and $Q$ is the observability Gramian, which is a function of $A_r$ and $C_r$, of \eqref{reduced}.

\subsection{Proof of \eqref{Lip_alpha}} \label{Sec3-B}

The function $f(A_r, B_r, C_r)$ with respect to $A_r$ is not convex even if $B_r$ and $C_r$ are fixed.
Due to this fact, it is difficult to derive the concrete expression of $L_{A_r}(B_r,C_r)$ unlike $L_{B_r}(A_r,C_r)$ and $L_{C_r}(A_r,B_r)$.

However, we can obtain the following theorem using the compactness of the set $S_{A_r}(A_r^{(0)},\epsilon,\gamma)$ defined in Theorem
\ref{Thm_key_set}.

\begin{theorem} \label{main_Thm}
For any $((A_r)_1,B_r,C_r),\, ((A_r)_2,B_r,C_r)\in S_{A_r}(A_r^{(0)},\epsilon,\gamma)\times {\bb R}^{r\times m}\times {\bb R}^{p\times r}$,
there exist positive constants $c_1$ and $c_2$
such that
\begin{align}
L_{A_r}(B_r,C_r)=(c_1+c_2\|B_r\|_{\rm F}\|C_r\|_{\rm F})\|B_r\|_{\rm F}\|C_r\|_{\rm F} \label{Lip_const_Ar}
\end{align}
satisfying \eqref{Lip_alpha}.
\end{theorem}
\begin{proof}
It suffices to show that there exist positive constants $c_1$ and $c_2$ such that \eqref{Lip_const_Ar} satisfies
\begin{align}
\|{\rm Hess}_{A_r} f(A_r,B_r, C_r)[\xi]\|_{\rm F} \leq L_{A_r}(B_r,C_r)\|\xi\|_{\rm F}, \label{Hessian_goal}
\end{align}
where ${\rm Hess}_{A_r} f(A_r,B_r, C_r)[\xi]$ is defined as
\begin{align*}
&\nabla_{A_r} f(A_r+\xi,B_r, C_r)-\nabla_{A_r} f(A_r,B_r, C_r) \\
=&{\rm Hess}_{A_r} f(A_r,B_r, C_r)[\xi]+o(\|\xi\|_{\rm F}).
\end{align*}
From \eqref{grad_h_A}, we obtain
\begin{align}
    &{\rm Hess}_{A_r} f(A_r,B_r, C_r)[\xi] \nonumber\\
    =& Q'P+QP'+Y'^\top X+Y^\top X', \label{Hess_f}
\end{align}
where $X'$, $Y'$, $P'$, and $Q'$ are the derivatives of $X$, $Y$, $P$, $Q$ in the direction of $\xi$, respectively, in terms of $A_r$.
That is, it follows from \eqref{X}, \eqref{Y}, \eqref{P}, and \eqref{Q}
that $X'$, $Y'$, $P'$, and $Q'$ are the solutions to the Sylvester equations
\begin{align*}
    AX'+X'A_r^{\top}+X\xi^\top &= 0, \\
    A^{\top}Y'+Y'A_r+Y\xi &=0, \\
    A_rP'+P'A_r^{\top}+\xi P +P\xi^{\top} &= 0, \\
    A_r^{\top}Q'+Q'A_r + \xi^{\top}Q+Q\xi &=0,
\end{align*}
respectively.
Because $A$ is stable and $A_r$ is also stable on the compact set $S_{A_r}(A_r^{(0)},\epsilon,\gamma)$,
the integral formulas to the Sylvester equations guarantee that there exist positive constants $c_X$, $c_Y$, $c_P$, $c_Q$, $c_{X'}$, $c_{Y'}$, $c_{P'}$, and $c_{Q'}$ such that
\begin{align*}
    &\|X\|_{\rm F}\leq c_X \|B_r\|_{\rm F},\,\,
    \|Y\|_{\rm F}\leq c_Y \|C_r\|_{\rm F},\\
    &\|P\|_{\rm F}\leq c_P \|B_r\|^2_{\rm F}, \,\,
    \|Q\|_{\rm F}\leq c_Q \|C_r\|^2_{\rm F},\\
    &\|X'\|_{\rm F}\leq c_{X'}\|X\|_{\rm F}\|\xi\|_{\rm F},\,\,
    \|Y'\|_{\rm F}\leq c_{Y'}\|Y\|_{\rm F}\|\xi\|_{\rm F},\\
    &\|P'\|_{\rm F}\leq c_{P'}\|P\|_{\rm F}\|\xi\|_{\rm F},\,\,
    \|Q'\|_{\rm F}\leq c_{Q'}\|Q\|_{\rm F}\|\xi\|_{\rm F}.
\end{align*}
Because \eqref{Hess_f} implies
\begin{align*}
    &\|{\rm Hess}_{A_r} f(A_r,B_r,C_r)[\xi]\|_{\rm F}\\
    \leq & \|P\|_{\rm F} \|Q'\|_{\rm F} + \|Q\|_{\rm F} \|P'\|_{\rm F} + \|X\|_{\rm F} \|Y'\|_{\rm F} + \|Y\|_{\rm F} \|X'\|_{\rm F},
\end{align*}
we obtain \eqref{Hessian_goal},
where
$c_1:=c_Xc_Y(c_{X'}+c_{Y'})$ and $c_2:=c_Pc_Q(c_{P'}+c_{Q'})$.
This completes the proof.\qed
\end{proof}




\section{Algorithm for solving Problem 1} \label{Sec5}

Based on Theorems \ref{Thm1} and \ref{main_Thm}, we propose Algorithm 1,
which iteratively updates $A_r$, $B_r$, and $C_r$ using projected gradient methods.
Algorithm 1 always generates a stable positive reduced network system \eqref{reduced},
which is better than the initial model \eqref{reduced_initial} in the sense of the $H^2$ norm.
This is because 
the map ${\rm proj}_{{\rm st}(A_r^{(0)})\cap S_{A_r}(A_r^{(0)},\epsilon,\gamma)}(A_r)$
is the projection onto the compact convex set ${\rm st}(A_r^{(0)})\cap S_{A_r}(A_r^{(0)},\epsilon,\gamma)$
in ${\bb R}^{r\times r}$,
and
 the maps ${\rm proj}_{{\rm st}(B_r^{(0)})}(B_r)$ and ${\rm proj}_{{\rm st}(C_r^{(0)})}(C_r)$ are the projections onto the closed convex sets ${\rm st}(B_r^{(0)})$ and ${\rm st}(C_r^{(0)})$, respectively.
That is,
for $i\neq j$, 
\begin{align*}
&\left({\rm proj}_{{\rm st}(A_r^{(0)})\cap S_{A_r}(A_r^{(0)},\epsilon,\gamma)}(A_r)\right)_{ij}\\
=&\begin{cases}
0 \quad\quad {\rm if} \quad (\tilde{A}_r)_{ij}\leq 0, \\
(\tilde{A}_r)_{ij}\quad {\rm if} \quad 0<(\tilde{A}_r)_{ij}\leq (\bar{A}_r)_{ij}, \\
 (\bar{A}_r)_{ij}\quad {\rm if} \quad (\tilde{A}_r)_{ij} > (\bar{A}_{r})_{ij},
\end{cases}
\end{align*}
and
\begin{align*}
&\left({\rm proj}_{{\rm st}(A_r^{(0)})\cap S_{A_r}(A_r^{(0)},\epsilon,\gamma)}(A_r)\right)_{ii}\\
=&\begin{cases}
-\gamma \quad\quad {\rm if} \quad (A_r)_{ii}< -\gamma, \\
(A_r)_{ii}\quad {\rm if} \quad -\gamma\leq (A_r)_{ii}\leq (\bar{A}_r)_{ii}, \\
 (\bar{A}_r)_{ii}\quad {\rm if} \quad (A_r)_{ii} > (\bar{A}_{r})_{ii},
\end{cases}
\end{align*}
where $\bar{A}_r$ is defined as \eqref{bar_Ar} and
\begin{align*}
\left(\tilde{A}_r\right)_{ij}
:=\left({\rm proj}_{{\rm st}(A_r^{(0)})}(A_r)\right)_{ij}
=
\begin{cases}
0\quad {\rm if}\quad (i,j)\in {\rm z} (A_r^{(0)}), \\
(A_r)_{ij}\quad {\rm otherwise}.
\end{cases}
\end{align*}
Moreover, for $M_r\in \{B_r, C_r\}$,
\begin{align*}
    &\left({\rm proj}_{{\rm st}(M_r^{(0)})}(M_r)\right)_{ij} \\
    =&
    \begin{cases}
    0\quad \quad\quad\, {\rm if}\quad (i,j)\in {\rm z} (M_r^{(0)}), \\
        0 \quad \quad\quad\, {\rm if}\quad (M_r)_{ij}< 0,\\
    (M_r)_{ij} \quad {\rm if}\quad (i,j)\not\in {\rm z} (M_r^{(0)})\,\,{\rm and}\,\, (M_r)_{ij}\geq 0,
\end{cases}
\end{align*}

\begin{algorithm*} 
\caption{Cyclic block projected gradient method.}         
\begin{algorithmic}[1]
\REQUIRE $(A,B,C)\in {\bb R}^{n\times n}\times {\bb R}^{n\times m}\times {\bb R}^{p\times n}$ in \eqref{original},
$(A_r^{(0)}, B_r^{(0)}, C_r^{(0)})\in {\bb R}^{r\times r}\times {\bb R}^{r\times m}\times {\bb R}^{p\times r}$ in \eqref{shokiten}, $\epsilon>0$, $\gamma>0$, $c>1$, $c_1>0$, $c_2>0$, and $k\leftarrow 0$.

\ENSURE $(A_r^{(k)},B_r^{(k)},C_r^{(k)})\in {\bb R}^{r\times r}\times {\bb R}^{r\times m}\times {\bb R}^{p\times r}$.

\REPEAT

\STATE Calculate $c_{A_r}:=c L_{A_r}(B_r^{(k)},C_r^{(k)})$, where $L_{A_r}(B_r,C_r)$ is defined in \eqref{Lip_const_Ar}.

\STATE $A_r^{(k+1)} = {\rm proj}_{{\rm st}(A_r^{(0)})\cap S_{A_r}(A_r^{(0)},\epsilon,\gamma)} \left(A_r^{(k)} - \nabla_{A_r}f(A_r^{(k)}, B_r^{(k)}, C_r^{(k)})/c_{A_r}\right)$.

\STATE Calculate $c_{B_r}:=c L_{B_r}(A_r^{(k+1)},C_r^{(k)})$, where $L_{B_r}(A_r,C_r)$ is defined in \eqref{Lip_B2}.

\STATE $B_r^{(k+1)} = {\rm proj}_{{\rm st}(B_r^{(0)})} \left(B_r^{(k)} - \nabla_{B_r}f(A_r^{(k+1)}, B_r^{(k)}, C_r^{(k)})/c_{B_r}\right)$.

\STATE Calculate $c_{C_r}:=c L_{C_r}(A_r^{(k+1)},B_r^{(k+1)})$, where $L_{C_r}(A_r,B_r)$ is defined in \eqref{Lip_C2}.

\STATE $C_r^{(k+1)} = {\rm proj}_{{\rm st}(C_r^{(0)})} \left(C_r^{(k)} - \nabla_{C_r}f(A_r^{(k+1)}, B_r^{(k+1)}, C_r^{(k)})/c_{C_r}\right)$.

\STATE $k\leftarrow k+1$.
\UNTIL{$(A_r^{(k)},B_r^{(k)},C_r^{(k)})$ is sufficiently close to a stationary point of Problem 1.}
\end{algorithmic}
\end{algorithm*}

In Algorithm 1, to calculate the gradients of $f$ in terms of $A_r$, $B_r$, and $C_r$,
we use the solutions $X$ and $Y$ to Sylvester equations \eqref{X} and \eqref{Y}.
If the original matrix $A$ is sparse,
we can use an efficient method whose computational complexity is considerably smaller than $O(n^3)$  for solving \eqref{X} and \eqref{Y},
as explained in Section 3 in \cite{benner2011sparse} and
Section 4 in \cite{simoncini2016computational}.
Note that even if $A$ is not sparse,
we can solve \eqref{X} and \eqref{Y} with the 
the computational costs $O(n^3)$ using
the Bartels--Stewart method
proposed in \cite{bartels1972solution}.

The following convergence property of Algorithm 1 can be easily proved using Theorems \ref{Thm1} and
\ref{main_Thm},
because Algorithm 1 is a special case of the proximal alternating linearized minimization proposed in
\cite{bolte2014proximal}.
That is, we can confirm that a slightly modified condition, which is needed to prove the convergence property, of Assumption 2 in \cite{bolte2014proximal} holds.

\begin{theorem} \label{Thm_convergence}
Suppose that $\{(A_r^{(k)}, B_r^{(k)}, C_r^{(k)})\}$ is a bounded, controllable, and observable sequence generated by Algorithm 1,
where we choose $c_1$ and $c_2$ such that the statement of Theorem \ref{main_Thm}
holds.
Then, 
 $\{(A_r^{(k)}, B_r^{(k)}, C_r^{(k)})\}$ converges to a stationary point of
Problem 1', that is, Problem 1.
\end{theorem}

\begin{remark} \label{remark_line_search}
In Algorithm 1, the step sizes for $A_r$, $B_r$, and $C_r$ updates can be defined using the Lipschitz constants derived in Section \ref{Sec4} without a line-search.
This is practically important,
because the computational cost of the line search is high due to the need for {calculating} the solution to large-scale Sylvester equations \eqref{X} or \eqref{Y}.
\end{remark}

\section{Numerical Experiments} \label{Sec6}

In this section, we demonstrate
the effectiveness of Algorithm 1
{with comparisons to the reduction method based on the RALM proposed in \cite{misawa2022h}.}
In all numerical experiments, we used $c=1.1$, $c_1=c_2=1$, $\epsilon=10^{-6}$, and $\gamma =10^7$ in Algorithm 1.

To this end, we considered the 2-dimensional heat equation on $[0,d]^2$
\begin{align*}
    \frac{\partial \theta}{\partial t}(x,y,t) = a \left(\frac{\partial^2 \theta}{\partial x^2}(x,y,t)+ \frac{\partial^2 \theta}{\partial y^2}(x,y,t)\right),
\end{align*}
which has been used for the thermal analysis of a building brick \cite{alawadhi2008thermal}
    and a heated plate \cite{barron2014synchronization}, where $a$ is the thermal conductivity.
We set $d=10$ and $a=0.0241$.
In addition, the Dirichlet boundary conditions were used to specify the actuators.
The finite difference discretization on $[0,d]^2$ of step size $h=d/(K+1)$ resulted in
system \eqref{original} with $n=K^2$, $m=p=2$, and
\begin{align}
A &:= \beta \begin{pmatrix}
A_1                         & I &                                &  & \\
I& A_1                         & I & & \\
                             & \ddots                     & \ddots                    &\ddots & \\
                             &                                &  I & A_1 & I \\
 & & & I & A_1              
\end{pmatrix}, \label{A_example}\\
A_1 &: = \begin{pmatrix}
-4                         & 1 &                                &  & \\
1& -4                         & 1 & & \\
                             & \ddots                     & \ddots                    &\ddots & \\
                             &                                & 1 & -4 & 1 \\
 & & & 1 & -4                         
\end{pmatrix}\in {\bb R}^{K\times K}, \nonumber \\
B &:= \beta\begin{pmatrix}
1 & 0 \\
0 & 0 \\
\vdots & \vdots \\
0 & 1
\end{pmatrix},\,
C := \begin{pmatrix}
1 & 0 & \cdots & 0 \\
0 & 0 & \cdots & 1
\end{pmatrix}, \label{BC_example}
\end{align}
under the assumption that we can directly control and measure $x_1(t)$ and $x_{n}(t)$ in Fig.\,\ref{Fig_PDE}.
The symmetric matrix $A$ is negative define, because $A$ is a diagonally dominant matrix with negative diagonal elements
\cite{saad2003iterative}.
Here, the above $I$ denotes the $K\times K$ identity matrix,
and $\beta:=a/h^2$.

Throughout all numerical experiments, we reduced ASPN \eqref{original} with \eqref{A_example} and \eqref{BC_example} to initial reduced ASPN \eqref{reduced_initial} with the reduced state dimension $r=16$.
The reduced ASPN has the interconnection structure, illustrated in Fig.\,\ref{Fig_clustering}.
Here, each cluster denotes an aggregated state variable composed of $K^2/16$ original states.

\begin{figure}[t]
  \begin{minipage}{0.52\columnwidth}
\begin{center}
\includegraphics[width = 4.6cm]{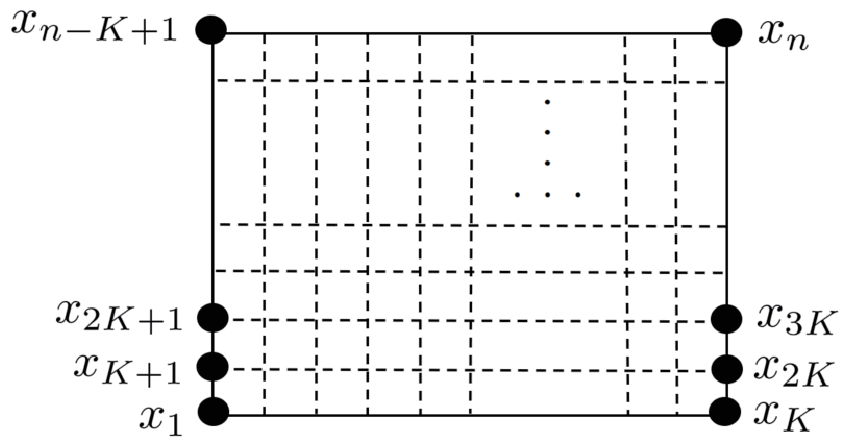}
\end{center}
\caption{Discretization of $[0,d]^2$.} 
\label{Fig_PDE}
\end{minipage}
  \begin{minipage}{0.45\columnwidth}
\begin{center}
\includegraphics[width = 3cm]{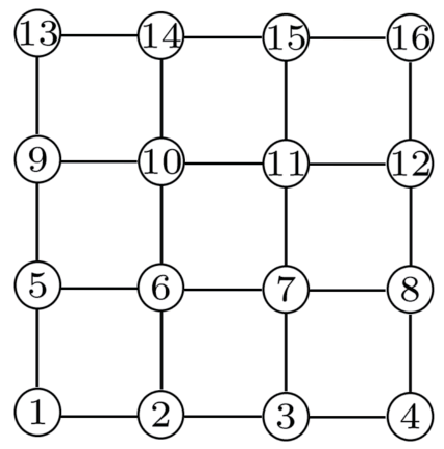} \\
\end{center}
\caption{Interconnection structure of reduced systems.}
\label{Fig_clustering}
\end{minipage}
\end{figure}

Table \ref{table1} shows the comparison of computational time between Algorithm 1 and {the RALM-based reduction method proposed in \cite{misawa2022h}}.
Here, the {iteration numbers of both methods were $10$.} 
The blank column for {the RALM-based reduction method } means that the calculation was not finished in {two days}.
The table indicates that even if the original state dimension $n$ is larger than $10^6$, Algorithm 1 can produce an ASPN with the interconnection structure in Fig.\,\ref{Fig_clustering} in a practical time period.
In contrast, we cannot expect that {the RALM-based reduction method} can produce a reduced system in a practical time period when $n$ is larger than {$1.6\times 10^5$}.

\begin{table}[t]
\caption{Computational time (in seconds) of Algorithm 1 and {the RALM-based reduction method \cite{misawa2022h}}. } \label{table1}
  \begin{center}
    \begin{tabular}{|c|c|c|c|} \hline
                 $n$ & $10^4$  & $1.6\times 10^5$ & $10^6$  \\ \hline 
      Algorithm 1    & $1.06\times 10^1$  & $2.52\times 10^2$  & $1.85\times 10^3$  \\ \hline 
      {RALM \cite{misawa2022h}}   &  {$1.59\times 10^3$}  & {$3.66\times 10^4$}  &  \\\hline 
    \end{tabular}
  \end{center}
\end{table}

{
Fig.\,\ref{Fig_objective} denotes the convergence behaviors of Algorithm 1 and the RALM-based reduction method
 when $n=10^4$.
After $10$ iterations, the proposed method for solving Problem 1 produced 
a better solution in terms of the $H^2$ norm compared with that of the RALM-based reduction method for solving Problem 0.
That is, although the feasible solution set of Problem 1 is narrower than that of Problem 0, this may not be an issue when we use the same initial point in Problems 0 and 1 due to the high non-convexity of Problem 0.
}

{Moreover, the RALM-based reduction method did not exactly produce an ASPN unlike Algorithm 1.
That is, the method could not preserve the interconnection structure unlike Algorithm 1. 
This means that the method generated an infeasible solution to Problem 0 at each iteration.  
The method may produce a feasible solution to Problem 0, if we set a sufficiently large iteration number or appropriately adjust hyper-parameters. However, in this case, we cannot obtain a reduced model in a practical time, as can be seen in Table \ref{table1}.}

\begin{figure}[t]
\begin{center}
\includegraphics[width = 8cm, height = 5cm]{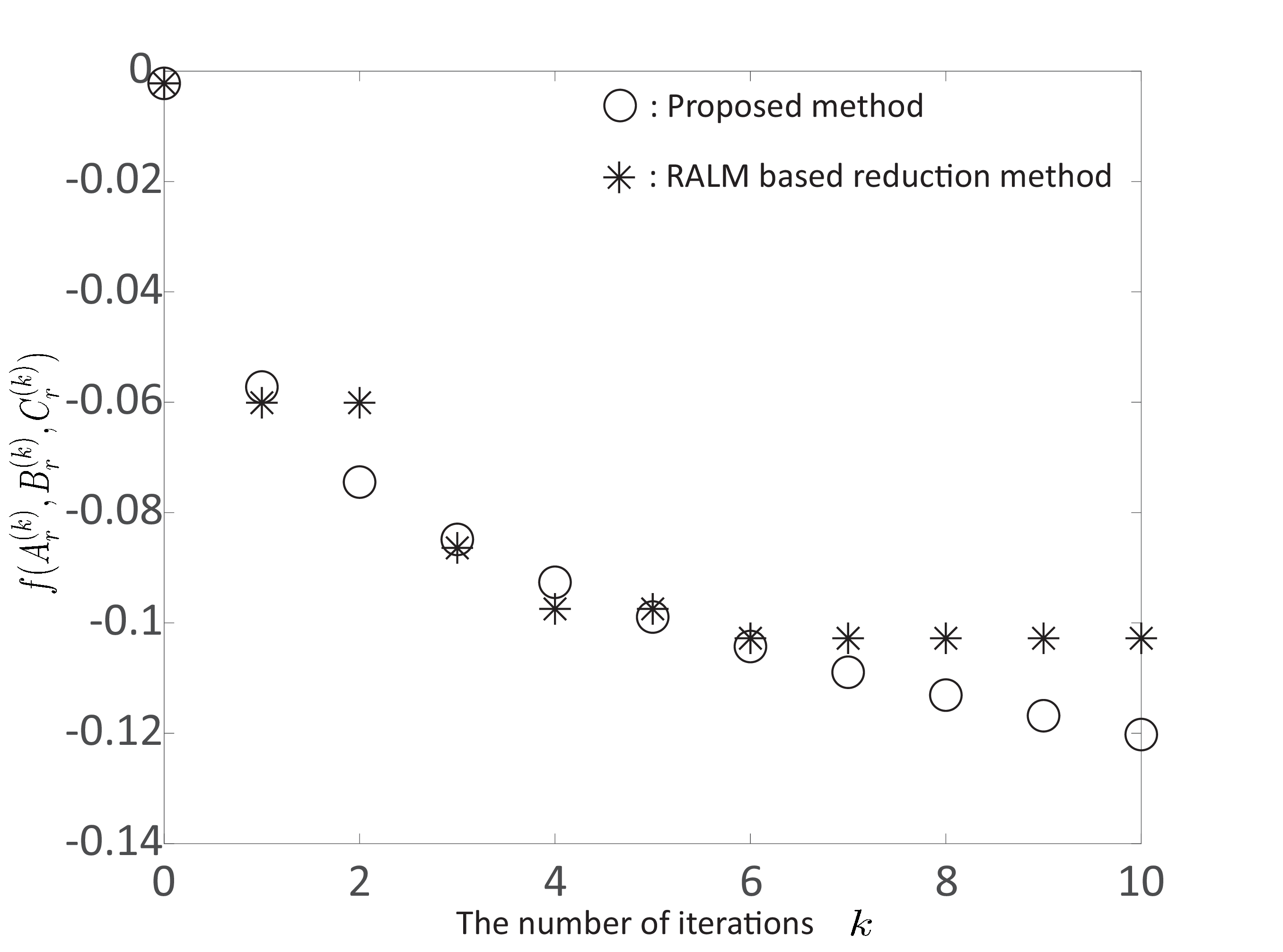}
\end{center}
\caption{{The comparison between Algorithm 1 and the RALM-based reduction method when $n=10^4$.}} \label{Fig_objective}
\end{figure}

\section{Concluding Remarks} \label{Sec7}
We proposed a reconstruction method, that preserves the stability, positivity, and original interconnection structure, for improving a reduced model generated using any reduction method by formulating a novel $H^2$ optimal model reduction problem with constraints.
To formulate the problem, we derived the set whose element is stable and Metzler.
In the proposed method, we used three Lipschitz constants, which analytically derived in this paper, of the gradients of our objective function to define the step sizes and guarantee the global convergence to a stationary point.
Moreover, in the numerical experiments, the proposed algorithm could generate a reduced model even if {the} original system {was} large-scale.

The proposed method in this paper for stable positive network systems \eqref{original} can be used to reduce semi-stable positive network systems.
In fact, because the output to semi-stable positive system \eqref{original} with $x(0)=0$ and $u(\tau)=\tilde{u}(\tau) e^{-\alpha(t-\tau)}$
is given by
    $y(t) = C\int_0^t \exp (A-\alpha {I_n})(t-\tau) B \tilde{u}(\tau) d\tau$.
Thus, instead of \eqref{relation_daiji}, for any $\alpha>0$,
we obtain
\begin{align*}
\sup_{t\geq 0}\|y(t)-y_r(t)\|_2 \leq \|\tilde{G}-\tilde{G}_r\|_{H^2}
\end{align*}
under $\|\tilde{u}\|_{L^2}\leq 1$,
where
$\tilde{G}$ is the transfer function of
\begin{align}
\begin{cases}
\dot{x}(t) = (A-\alpha I_n)x(t) + B\tilde{u}(t), \\
y(t) = Cx(t)
\end{cases} \label{original2}
\end{align}
and $\tilde{G}_r$ is the transfer function of system \eqref{reduced} with input $\tilde{u}$.
Even if positive system \eqref{original} is semi-stable,
the modified positive system \eqref{original2} is stable for any $\alpha>0$.
Thus, by using Algorithm 1 to \eqref{original2} with sufficiently small $\alpha >0$,
we can obtain reduced stable positive system \eqref{reduced} with input $\tilde{u}$.
To obtain a reduced semi-stable positive system,
the eigenvalue of $A_r$ should be shifted such that the largest eigenvalue of the modified $A_r$ is $0$.
Note that the shift can be performed, because Proposition \ref{Pro_general_nonnegative} in Appendix \ref{Ape_Perron_Frobenius} holds.

However, for a Laplacian dynamical system, which is a special class of semi-stable positive systems \cite{cheng2020clustering, cheng2020reduced, cheng2021model},
the above method cannot preserve the Laplacian dynamical structure, that is, the linear constraint of elements of $A$ matrix.
Thus, to reduce a large-scale Laplacian dynamical system to a small-scale Laplacian dynamical system, we need to modify our problem with the linear constraint.
In this case, we have to consider an adequate algorithm for the modified problem.
This is an interesting direction of future studies.


%


\section*{Acknowledgment}
This work was supported by Japan
Society for the Promotion of Science KAKENHI under Grant 20K14760. 

\appendix


\subsection{Proofs of Lemma \ref{Lem_Perron_Frobenius} and Theorem \ref{Thm_key_set}}

\subsubsection{Summaries of Perron--Frobenius theory} \label{Ape_Perron_Frobenius}

To prove Lemma \ref{Lem_Perron_Frobenius} and Theorem \ref{Thm_key_set},
we briefly summarize the Perron--Frobenius theory.

Let $A\in {\bb R}^{n\times n}$.
The spectral radius $\rho (A)$ of $A$ is defined as
$\rho (A):= \max \{ |\lambda|\,|\, \lambda \in \sigma (A)\}$,
where $\sigma (A)$ denotes the set of all eigenvalues of $A$.

As shown in Corollary 8.1.19 in \cite{horn2012matrix}, the following proposition holds.
\begin{proposition} \label{Pro_nonnegative_spectral}
Suppose that $A, B\in {\bb R}^{n\times n}$ are nonnegative.
If $A\leq B$, then
$\rho (A)\leq \rho(B)$ holds.
\end{proposition}

As shown in Chapter 8 in \cite{meyer2000matrix}, the following proposition holds for general nonnegative matrices.
\begin{proposition} \label{Pro_general_nonnegative}
If $A\in {\bb R}^{n\times n}$ is nonnegative, $\rho(A)$ is an eigenvalue of $A$.
\end{proposition}
Note that the spectral radius of any real matrix is not always an eigenvalue of the matrix.

A nonnegative matrix $A\in {\bb R}^{n\times n}$ is termed irreducible if the graph corresponding to $A$ is strongly connected.

The following proposition is a part of Perron--Frobenius theory, as shown in Chapter 8 in \cite{meyer2000matrix}.

\begin{proposition} \label{Pro_Perron_Frobenius}
Suppose that $A\in {\bb R}^{n\times n}$  is a nonnegative irreducible matrix.
The following statements hold.
\begin{enumerate}
\item 
The spectral radius $\rho (A)$ of $A$ is positive and is an algebraically simple eigenvalue of $A$.

\item There are the unique positive vectors $x, y\in {\bb R}^n$ such that
\begin{align}
Ax = \rho(A) x,\quad y^{\top}A= \rho(A) y^{\top},\quad y^{\top}x=1. \label{Pro_PF_daiji}
\end{align}

\item 
There are no nonegative right and left eigenvectors for $A$ except for positive multiples of $x$ and $y$ in \eqref{Pro_PF_daiji}.

\end{enumerate}
\end{proposition}

\subsubsection{Proof of Lemma \ref{Lem_Perron_Frobenius}} \label{Ape_proof_Lemma}

There exists $\alpha>0$ such that $A_r^{(0)}+\alpha I_r$ is nonnegative and irreducible, because we have assumed that $A_r^{(0)}$ is Metzler and irreducible. Thus, 
1) of Proposition \ref{Pro_Perron_Frobenius} in Appendix \ref{Ape_Perron_Frobenius}
 implies that the spectral radius $\rho$ of $A_r^{(0)}+\alpha I_r$ is an algebraically simple
 eigenvalue of $A_r^{(0)}+\alpha I_r$.
Thus, $\mu_1 := \rho- \alpha$ is a real number and \eqref{eigen_ineq} holds.
Here, $\mu_1<0$ follows from the assumption that $A_r^{(0)}$ is stable.
Moreover, 2) of Proposition \ref{Pro_Perron_Frobenius} implies that 
  the right eigenvector $v_1$ and left eigenvector $w_1$ corresponding to $\mu_1$ of $A_r^{(0)}$ are positive vectors, and \eqref{w_normalization} holds.
\qed


\subsubsection{Proof of Theorem \ref{Thm_key_set}} \label{Ape_proof_Theorem}

By the definition of $S_{A_r}(A_r^{(0)},\epsilon,\gamma)$, each matrix in $S_{A_r}(A_r^{(0)},\epsilon,\gamma)$ is a Metzler matrix.
The statement $A_r^{(0)}\in S_{A_r}(A_r^{(0)},\epsilon,\gamma)$ is also obvious. In fact, it follows
from the nonnegativity of $v_1 w_1^{\top}$ and $\mu_1+\epsilon\leq 0$ that 
$A_r^{(0)} \leq \bar{A}_r$
and $-\gamma I_r \leq A_r^{(0)}$ holds by the assumption of $-\gamma I_n \leq A_r^{(0)}$.

To show that each matrix in $S_{A_r}(A_r^{(0)},\epsilon,\gamma)$ is stable,
we note that Proposition \ref{Pro_nonnegative_spectral} in Appendix \ref{Ape_Perron_Frobenius} yields
\begin{align}
\rho(A_1) \leq \rho ({A}_2), \label{prove_stable}
\end{align}
for any $A_r\in S_{A_r}(A_r^{(0)},\epsilon,\gamma)$,
where 
$A_1:=A_r + \gamma I_r$ and $A_2:=\bar{A}_r+\gamma I_r$.
From Proposition \ref{Pro_general_nonnegative} in Appendix \ref{Ape_Perron_Frobenius},
 $\rho (A_1)$ is an eigenvalue of $A_1$.
Thus, 
\begin{align}
\mu := \rho(A_1)-\gamma \label{prove_stable2}
\end{align}
 is an eigenvalue with the largest real part of $A_1$. 
It follows from \eqref{prove_stable} and \eqref{prove_stable2} that
\begin{align}
\mu \leq \rho (A_2) -\gamma. \label{prove_stable3}
\end{align}
Moreover, 
\begin{align}
\rho (A_2) = -\epsilon + \gamma. \label{prove_stable4}
\end{align}
This follows from 
\begin{align}
A_2 v_1 = (-\epsilon + \gamma) v_1, \label{prove_stable5}
\end{align}
where we used the definitions of $A_2$, $\bar{A}_r$, and \eqref{eigen_ineq}.
In fact, from \eqref{prove_stable5}
and $v_1>0$, 2) and 3) in Proposition \ref{Pro_Perron_Frobenius} in Appendix \ref{Ape_Perron_Frobenius} imply \eqref{prove_stable4}.
Thus, by combining \eqref{prove_stable3} and \eqref{prove_stable4},
we obtain
$\mu \leq -\epsilon$.
This means that real parts of all the eigenvalues of any $A_r\in S_{A_r}(A_r^{(0)},\epsilon,\gamma)$ are less than or equal to $-\epsilon<0$, and thus any $A_r\in S_{A_r}(A_r^{(0)},\epsilon,\gamma)$ is stable.
This completes the proof. \qed

\ifCLASSOPTIONcaptionsoff
  \newpage
\fi



\bibliographystyle{IEEEtran}
\bibliography{bib.bib}




%





\end{document}